\documentclass[a4paper]{amsart}

\usepackage[english]{babel} 

\usepackage{aeguill} 

\usepackage{tikz}
\usepackage{tikz-cd}
\usepackage{amsmath}
\usepackage{amssymb}
\usepackage{mathrsfs}
\usepackage{bm}
\usepackage{mathtools}
\usepackage{pigpen}
\usepackage{picture}
\usepackage{hyperref}
\usepackage{url}
\usepackage{graphicx}
\usepackage{csquotes}
\usepackage{cmll}

\makeatletter
\tagsleft@false
\makeatother

\numberwithin{equation}{section}
\newcommand{\normono}{\mathbin{\tikz[baseline] \draw[{Triangle[open,reversed,length=1.5mm,width=1.5mm]}->] (0pt,0.75ex) -- (2.5ex,0.75ex);}}
\newcommand{\norepi}{\mathbin{\tikz[baseline] \draw[-{Triangle[open,length=1.5mm,width=1.5mm]}] (0pt,0.75ex) -- (2.5ex,0.75ex);}}
\newcommand{\pullback}{\text{\pigpenfont~A}}
\newcommand{\pushout}{\text{\pigpenfont~I}}

\newcommand{\ch}{\mathsf{Ch}}
\newcommand{\simp}{\mathsf{Simp}}
\newcommand{\dsimp}{\mathsf{DSimp}}
\newcommand{\deltacat}{\mathsf{\Delta}}

\renewcommand{\ker}{\mathrm{ker}}
\newcommand{\coker}{\mathrm{coker}}
\newcommand{\im}{\mathrm{im}}

\newcommand{\op}{\mathrm{op}}

\newcommand{\map}{-{Straight Barb[length=2mm,width=2mm]}}
\newcommand{\mono}{{Straight Barb[reversed,length=2mm,width=2mm]}-{Straight Barb[length=2mm,width=2mm]}}
\newcommand{\epi}{-Straight Barb[length=2mm,width=2mm] Straight Barb[length=2mm,width=2mm]}
\newcommand{\normalmono}{{Triangle[open,reversed,length=2mm,width=2mm]}-{Straight Barb[length=2mm,width=2mm]}}
\newcommand{\normalepi}{-{Triangle[open,length=2mm,width=2mm]}}
\newcommand{\normaldemimono}{{Triangle[open,reversed,length=2mm,width=2mm]}-}

\newcommand{\C}{\mathcal{C}}
\newcommand{\N}{\mathcal{N}}
\newcommand{\E}{\mathcal{E}}
\newcommand{\Z}{\mathbb{Z}}

\newcommand{\noproof}{\hfill$\square$}

\newcommand{\ie}{\textit{i.e.}\ }

\newcommand{\Def}[1]{\textit{\textbf{#1}}}

\theoremstyle{plain}
\newtheorem{proposition}{Proposition}[section]

\newtheorem{theorem}[proposition]{Theorem}
\newtheorem{lemma}[proposition]{Lemma}
\newtheorem*{theorem*}{Theorem}

\theoremstyle{definition}
\newtheorem{definition}[proposition]{Definition}

\theoremstyle{remark}

\title{Spectral Sequences in Semi-Abelian Categories}

\author{Florent Afsa}



\address{Institut de Recherche en Math\'ematique et Physique, Universit\'e catholique de Louvain, Chemin du Cyclotron 2, 1348 Louvain-la-Neuve, Belgium}

\email{florent.afsa@uclouvain.be}

\begin{document}

\begin{abstract}
	In this paper, we define the notion of a spectral sequence in the context of semi-abelian categories in the sense of Janelidze, Márki and Tholen. We show that from a double simplicial object, one can construct an exact couple, which gives rise to a spectral sequence. This extends Quillen's result on double simplicial groups to every semi-abelian category.
\end{abstract}

\maketitle

\section*{Introduction}

In 1965, Quillen proved in \cite{DGQlln1966} that for a double semi-simplicial group $G$, there are  spectral sequences
\begin{align*}
	E^2_{p,q}=H_p^hH_q^vG\Rightarrow H_{p+q}(\Delta G) \\
	E^2_{p,q}=H_p^vH_q^hG\Rightarrow H_{p+q}(\Delta G)
\end{align*}
This was already known in the case of abelian groups \cite[Satz 2.9]{DoldPuppe}. The main result of our current paper is Theorem \ref{main theorem}, which extends Quillen's result to every semi-abelian category \cite{Janelidze-Marki-Tholen} (but only for double simplicial objects). Semi-abelian categories are pointed categories which are exact in the sense of Barr \cite{Barr}, protomodular in the sense of Bourn \cite{Bourn:protomodular}, and admit binary coproducts. These satisfy some algebraic properties and enable one to do homological algebra in a wider context than abelian categories. Some examples are:
\begin{itemize}
	\item Every abelian category;
	\item The category of groups;
	\item The category of Lie algebras over any ring;
	\item The category of rings with or without unit;
	\item The category of crossed modules;
	\item Actually, any variety of $\Omega$-groups \cite{Higgins}, \ie variety of universal algebras with underlying group structure such that the trivial subgroup is a subalgebra, which includes the four previous examples;
	\item The category of cocommutative Hopf algebras over a field of characteristic $0$ \cite{GranKadjoVercruysse}.
\end{itemize}
There is a direct application in \cite{AfsaDonadzeVdLinden}, where the authors explore sufficient conditions for higher \v{C}ech derived functors to coincide with ordinary derived functors. Our Theorem \ref{main theorem} is used there to prove a categorical version of \textit{Loday's theorem} - essentially  Proposition 3.4 in \cite{Loday}, of which an algebraic version is Proposition 14 of \cite{DonadzeInaPorter} - stated as follows:
\begin{theorem*}\cite{AfsaDonadzeVdLinden}
	In a semi-abelian category, given any internal $(n+1)$-fold groupoid $G$ of which the projections in direction $i\in k\leqslant n+1$ are written $d_0^i$, $d_1^i\colon G_k\rightarrow G_{k\backslash\{i\}}$, we have
	\[
		H_{n+1}\Delta(\mathrm{Ner}^{n+1}(G))=\underset{i\in n+1}{\bigcap}\big(\mathrm{Ker}(d_0^i)\cap\mathrm{Ker}(d_1^i)\big)
	\]
	and $H_j\Delta(\mathrm{Ner}^{n+1}(G))=0$ for all $j\geqslant n+2$.\noproof
\end{theorem*}
This allows the authors of \cite{AfsaDonadzeVdLinden} to obtain a Hopf formula for \v{C}ech derived functors.

Towards Theorem \ref{main theorem}, in Section \ref{section spectral sequences} we formalize the notion of a spectral sequence in the context of semi-abelian categories. Spectral sequences, introduced by Leray in \cite{Leray} for modules, help to understand the homology of a chain complex equipped with some additional structure, such as a filtration or a grading. A spectral sequence is a sequence of bigraded objects $(E^r)_r$ together with morphisms $d^r$ at each level, called \textit{differentials}, such that the homology of $E^r$ is isomorphic to $E^{r+1}$. One of the theory's main tools is the concept of an exact couple (Definition \ref{def exact couple}), originally introduced by Massey in \cite{Massey}. An exact couple is the data of two bigraded objects $D$ and $E$ and morphisms between them. This gives rise to a spectral sequence in which the first page is $E$.

The definitions and constructions we make here are not new, and have been considered in diverse contexts in the literature. For instance, spectral sequences and exact couples are defined in \cite[3.5]{Grandis} with different conventions and hypotheses in the context of \textit{Grandis homological categories}\footnote{Grandis homological categories are not to be confused with the homological categories of Borceux and Bourn \cite{Borceux-Bourn}.} \cite[1.3.6]{Grandis}, which are not necessarily pointed, but such that any composite of normal monomorphisms is again normal, which excludes the category of all groups. Our context of semi-abelian categories does not immediately compare to that one; some properties which hold in the abelian situation but not here make the setup more delicate. In particular, we observe that the construction of a spectral sequence from an exact couple is not automatic any more, and requires the hypothesis of the normality of the differentials.

Next, we construct an exact couple out of any double simplicial object. This follows Quillen's construction almost identically, except for Proposition \ref{prop homology EG}. To this end, we use the Moore complex (Definition \ref{def Moore complex}) of a simplicial object. The well know fact that the homology of a simplicial object is abelian (in all degrees except in degree $0$), recalled in Theorem \ref{theorem abelian homology object}, guarantees the needed normality of the differentials.\\

In the following, we let $\C$ denote a semi-abelian category \cite{Janelidze-Marki-Tholen}. Let $f\colon X\rightarrow Y$ be a morphism in $\C$. We write $f\colon X\normono Y$ (respectively $f\colon X\norepi Y$) if $f$ is a normal monomorphism, \ie a kernel of some morphism (respectively a normal epimorphism, \ie a cokernel of some morphism). We write $K(f)$ (respectively $C(f)$, $I(f)$) for a chosen representative of the kernel (respectively the cokernel, the image) of $f$, and $\ker(f)$ (respectively $\coker(f)$, $\im(f)$) for the morphism $K(f)\normono X$ (respectively $Y\norepi C(f)$, $I(f)\rightarrowtail Y$). $K$, $C$ and $I$ are functors from the category of arrows of $\C$ to the category $\C$, and $\ker$, $\coker$ and $\im$ are endofunctors of the category of arrows of $\C$. A morphism $f$ is called \Def{normal} if it factors as $f=me$ where $m$  is a normal monomorphism and $e$ is a normal epimorphism. We recall the fact that for a morphism $f$ in $\C$, $I(f)=C(\ker(f))$, and whenever $f$ is normal, $I(f)=K(\coker(f))$.

\section{Spectral Sequences}\label{section spectral sequences}

Spectral sequences have been introduced by Jean Leray in \cite{Leray} to compute sheaf cohomology. They have afterwards become a standard tool in homological algebra, and subsequently the basics of this theory have been grounded in the framework of abelian categories. In this section, we define spectral sequences in the context of semi-abelian categories, as well as one of its basic tools, namely exact couples, introduced by Massey in \cite{Massey}. Instead of the definition of a spectral sequence using morphisms with bidegrees (such as in \cite{Grandis}), we prefer to involve shift functors so we can stay in a pointed category.

\begin{definition}
	Let $\C$ be a semi-abelian category. A \Def{bigraded object} $E$ is a functor from the discrete category $\Z\times\Z$ to $\C$, \ie an indexed set of objects $E_{p,q}$ of $\C$ with $p$, $q\in\Z$. We obtain the category $\C^{\Z\times\Z}$ of bigraded objects which is again semi-abelian (as is any category of functors with a semi-abelian codomain).
\end{definition}

We define automorphisms $S$, $T\colon \C^{\Z\times\Z}\rightarrow\C^{\Z\times\Z}$ of the category of bigraded objects by setting $(SE)_{p,q}=E_{p-1,q+1}$ and $(TE)_{p,q}=E_{p,q-1}$ for each bigraded object $E$. Clearly, $S$ and $T$ commute. Since limits and colimits in functor categories are computed pointwise, these also commute with the functors $K$, $C$ and $I$, in the sense that for a morphism $f$, we have $K(Sf)=SK(f)$, for instance.

\begin{definition}
	Given an integer $a\geqslant 0$, a \Def{spectral sequence} $(E,d)$ (starting with $E^a$) consists of a sequence of bigraded objects $(E^r)_{r\geqslant a}$ together with morphisms $d^r\colon S^{-r}T^{-1}E^r\rightarrow E^r$. These have to satisfy the following conditions:
	\begin{itemize}
		\item $(S^rTd^r)d^r\colon S^{-r}T^{-1}E^r\longrightarrow S^rTE^r$ is the zero morphism, which gives a map $I(d^r)\longrightarrow K(S^rTd^r)$;
		\item $E^{r+1}\cong C\big(I(d^r)\longrightarrow K(S^rTd^r)\big)$; the differential of $E^{r+1}$ is therefore $d^{r+1}\colon S^{-r-1}T^{-1}E^{r+1}\rightarrow E^{r+1}$.
	\end{itemize}
	The object $E^r$ is called the $r$-th \Def{page} of $E$. The maps $d^r$ are called \Def{differentials}, and are componentwise $d^r_{p,q}\colon E^r_{p+r,q-r+1}\rightarrow E^r_{p,q}$.
\end{definition}

A spectral sequence starting with $E^a$ is \Def{bounded} whenever for each integer $n$ there are finitely many non-zero terms $E^a_{p,q}$ with $p+q=n$. When this happens, there exists $r_0\geqslant a$ such that for each $r\geqslant r_0$, $d^r_{p,q}=(S^rTd^r)_{p,q}=0$, and thus $E^r_{p,q}\cong E^{r_0}_{p,q}$ ($r^0$ depends of $p$ and $q$). We write $E^\infty_{p,q}$ for this stable value, and the resulting bigraded object $E^\infty$ is called the \Def{limit page} of the spectral sequence $(E,d)$.

\begin{definition}
	Let $(H_n)_{n\geqslant0}$ be a family of objects of $\C$. We say that a bounded spectral sequence $E$ \Def{converges} to $H_\ast$ if for each $n$ there is a finite filtration of $H_n$
	\[
		0=F_sH_n\leqslant\dots F_pH_n\leqslant F_{p+1}H_n\leqslant\dots\leqslant F_tH_n=H_n
	\]
	and isomorphisms $C(F_{p-1}H_n\rightarrow F_pH_n)\cong E^\infty_{p,q}$ where $p+q=n$. This is usually written as follows:
	\[
		E^a_{p,q}\Rightarrow H_{p+q}
	\]
\end{definition}

\begin{proposition}\label{prop Br Zr}
	Let $(E,d)$ be a spectral sequence. There exist two sequences of bigraded objects
	\[
		0=B^a\leqslant B^{a+1}\leqslant\dots\leqslant B^r\leqslant\dots\leqslant Z^r\leqslant\dots\leqslant Z^{a+1}\leqslant Z^a=E^a
	\]
	such that $E^r\cong C(B^r\rightarrow Z^r)$ for each $r\geqslant a$.
\end{proposition}
\begin{proof}
	Let $(E,d)$ be a spectral sequence. We define the two sequences inductively by setting $B^a=0$, $Z^a=E^a$ and the pullbacks
	\begin{equation}\label{diag Br Zr}
		\begin{tikzpicture}[baseline=(I)]
			\node(A) at (0,0){$B^{r+1}$};
			\node(B) at (4,0){$Z^{r+1}$};
			\node(C) at (8,0){$Z^r$};
			\node(D) at (0,-2){$I(d^r)$};
			\node(E) at (4,-2){$S^rTK(d^r)$};
			\node(F) at (8,-2){$E^r$};

			\node(G) at (0.5,-0.5){$\pullback$};
			\node(H) at (4.5,-0.5){$\pullback$};

			\node(I) at (-1,-1.1){};

			\draw[\mono] (A)--(B);
			\draw[\normalmono] (B)--(C);
			\draw[\normalepi] (A)--(D);
			\draw[\normalepi] (B)--(E);
			\draw[\normalepi] (C)--(F);
			\draw[\mono] (D)--(E);
			\draw[\normalmono] (E)--(F);
		\end{tikzpicture}
	\end{equation}
	The map $Z^r\norepi E^r$ is the cokernel of $B^r\rightarrow Z^r$. Since the left-hand square is a pullback, it is also a pushout by Lemma \ref{lemma pb with normal epis}. Therefore by Lemma \ref{lemma cokernel of pushout}, the cokernel $C\big(B^{r+1}\rightarrow Z^{r+1}\big)$ is isomorphic to $C\big(I(d^r)\longrightarrow S^rTK(d^r)\big)$, which is itself isomorphic to $E^{r+1}$.
\end{proof}

Whenever they exist, we define
\[
	B^\infty\coloneq\underset{r\geqslant a}{\bigcup}B^r \textrm{ and } Z^\infty\coloneq\underset{r\geqslant a}{\bigcap}Z^r
\]
If $(E,d)$ is bounded, then $(S^rTK(d^r))_{p,q}=E^r_{p,q}$ and $I(d^r)_{p,q}=0$ for large $r$, so $B^\infty_{p,q}=B^r_{p,q}$ and $Z^\infty_{p,q}=Z^r_{p,q}$. Thus we obtain $E^\infty\cong C(B^\infty\rightarrow Z^\infty)$.

\begin{definition}\label{def exact couple}
	An \Def{exact couple} $\E=(D,E,f,g,h)$ of degree $a$ in $\C$ consists of the following data:
	\begin{itemize}
		\item two bigraded objects $D$ and $E$ of $\C$;
		\item a morphism $f\colon SD\rightarrow D$;
		\item a morphism $g\colon D\rightarrow E$;
		\item a morphism $h\colon S^{-a}T^{-1}E\rightarrow D$.
	\end{itemize}
	These are required to fit into a long exact sequence
	\[
		\cdots\rightarrow S^{-a}T^{-1}D\xrightarrow{S^{-a}T^{-1}g}S^{-a}T^{-1}E\xrightarrow{h} D\xrightarrow{S^{-1}f} S^{-1}D\xrightarrow{S^{-1}g} S^{-1}E\rightarrow\cdots
	\]
	which is componentwise
	\[
		\cdots\rightarrow D_{p+a,q-a+1}\xrightarrow{g_{p+a,q-a+1}}E_{p+a,q-a+1}\xrightarrow{h_{p,q}} D_{p,q}\xrightarrow{f_{p+1,q-1}}D_{p+1,q-1}\xrightarrow{g_{p+1,q-1}}\cdots
	\]
	This is usually written as
	\[
		\begin{tikzpicture}
			\node(A) at (0,0){$D$};
			\node(B) at (3,0){$D$};
			\node(C) at (1.5,-2){$E$};

			\draw[\map] (A)--(B) node[above,midway] {$f$};
			\draw[\map] (B)--(C) node[below right,midway] {$g$};
			\draw[\map] (C)--(A) node[below left,midway] {$h$};
		\end{tikzpicture}
	\]
	(Here, a sequence of morphisms $X\overset{f}{\rightarrow}Y\overset{g}{\rightarrow}Z$ is \Def{exact} if $f$ factors through the kernel of $g$, and $g$ factors through the cokernel of $f$.)
\end{definition}

Marco Grandis already defines exact couples, but slightly differently than us \cite[Definition 3.5]{Grandis}. He does this in the context of \textit{Grandis homological categories} \cite[1.3.6]{Grandis}. A category is \Def{Grandis homological} if it has a closed ideal of null morphisms $\N$, admits kernels and cokernels with respect to $\N$, and is such that normal monomorphisms and normal epimorphisms are closed under composition, and for any normal monomorphism $m$ followed by a normal epimorphism $e$ such that $\ker e\leqslant m$, the composite $em$ is normal (here normality is defined with respect to $\N$). The pointed case where $\N$ consists of all zero maps almost comprises semi-abelianness, except that the latter does not require that normal monomorphisms compose. In \cite{hslat}, the condition that any composite of two normal monomorphisms is a normal monomorphism is called \Def{transitivity of normality}, and is explored in the framework of semi-abelian categories. In contrast, normal epimorphisms always compose in a semi-abelian category. The following theorem allows us to construct a spectral sequence from an exact couple, such as in the usual theory. However, the normality of the differential in every page is required. This morphism being defined as a composite of normal morphisms, its normality is automatic in the abelian case or in the setting of Grandis, but not in the semi-abelian case.

\begin{theorem}
	Given an exact couple $\E$ of degree $a$ as in Definition \ref{def exact couple}, we define the following data:
	\begin{itemize}
		\item $d\coloneq gh\colon S^{-a}T^{-1}E\rightarrow E$;
		\item $D'\coloneq S^{-1}I(f)$;
		\item $E'\coloneq C\big(I(d)\longrightarrow K(S^aTd)\big)$.
	\end{itemize}
	We assume that $d$ is a normal morphism. Then there exist morphisms $f'\colon SD'\rightarrow D'$, $g'\colon D'\rightarrow E'$, $h'\colon S^{-(a+1)}T^{-1}E'\rightarrow D'$ which fit into an exact couple $\E'$ of degree $a+1$
	\[
		\begin{tikzpicture}
			\node(A) at (0,0){$D'$};
			\node(B) at (3,0){$D'$};
			\node(C) at (1.5,-2){$E'$};

			\draw[\map] (A)--(B) node[above,midway] {$f'$};
			\draw[\map] (B)--(C) node[below right,midway] {$g'$};
			\draw[\map] (C)--(A) node[below left,midway] {$h'$};
		\end{tikzpicture}
	\]
	$\E'$ is called the \Def{derived exact couple} of $\E$.
\end{theorem}
\begin{proof}
	Since $(S^aTd)d=(S^aTg)(S^aTh)gh=0$, there is a factorization of $d$ as follows:
	\[
		\begin{tikzpicture}
			\node(A) at (0,0){$S^{-a}T^{-1}E$};
			\node(B) at (3,0){$E$};
			\node(C) at (0,-2){$I(d)$};
			\node(D) at (3,-2){$S^aTK(d)$};

			\draw[\map] (A)--(B) node[above,midway] {$d$};
			\draw[\normalmono] (C)--(D) node[below,midway] {$\bar{d}$};
			\draw[\normalepi] (A)--(C) node[left,midway] {$p_d$};
			\draw[\normalmono] (D)--(B) node[right,midway] {$k$};
		\end{tikzpicture}
	\]
	Therefore, we can define $E'=C(\bar{d})=C(\bar{d}p_d)$, and we denote $q\colon K(S^aTd)\norepi E'$. By definition of an exact couple, there are normal factorizations $f=m_fp_f$, $g=m_gp_g$, $h=m_hp_h$ and short exact sequences
	\[
		\begin{tikzpicture}
			\node(A) at (0,0){$I(g)$};
			\node(B) at (3,0){$E$};
			\node(C) at (6.33,0){$S^aTI(h)$};

			\draw[\normalmono] (A)--(B) node[above,midway] {$m_g$};
			\draw[\normalepi] (B)--(C) node[above,midway] {$S^aTp_h$};

			\node(D) at (0,-1){$I(h)$};
			\node(E) at (3,-1){$D$};
			\node(F) at (6.32,-1){$S^{-1}I(f)$};

			\draw[\normalmono] (D)--(E) node[above,midway] {$m_h$};
			\draw[\normalepi] (E)--(F) node[above,midway] {$S^{-1}p_f$};

			\node(G) at (0,-2){$I(f)$};
			\node(H) at (3,-2){$D$};
			\node(I) at (6,-2){$I(g)$};

			\draw[\normalmono] (G)--(H) node[above,midway] {$m_f$};
			\draw[\normalepi] (H)--(I) node[above,midway] {$p_g$};
		\end{tikzpicture}
	\]
	Let $f'\coloneq (S^{-1}p_f)m_f\colon SD'\rightarrow D'$. Since $m_f=\ker(g)$ and $S^{-a}T^{-1}k=\ker(gh)$, by Lemma \ref{lemma kernel of composite} we have a pullback square
	\[
		\begin{tikzpicture}
			\node(A) at (0,0){$K(d)$};
			\node(B) at (3,0){$SD'$};
			\node(C) at (0,-2){$S^{-a}T^{-1}E$};
			\node(D) at (3,-2){$D$};

			\node(E) at (0.5,-0.5){$\pullback$};

			\draw[\map] (A)--(B) node[above,midway] {$\bar{h}$};
			\draw[\map] (C)--(D) node[below,midway] {$h$};
			\draw[\normalmono] (A)--(C) node[left,midway] {$S^{-a}T^{-1}k$};
			\draw[\normalmono] (B)--(D) node[right,midway] {$m_f$};
		\end{tikzpicture}
	\]
	Moreover, we have
	\[
		m_f\bar{h}(S^{-a}T^{-1}\bar{d})(S^{-a}T^{-1}p_d)=h(S^{-a}T^{-1}d)=0
	\]
	hence $\bar{h}(S^{-a}T^{-1}\bar{d})=0$, so there exists $h'\colon S^{-(a+1)}T^{-1}E'\rightarrow D'$ such that $S^{-1}\bar{h}=h'(S^{-(a+1)}T^{-1}q)$. We also have
	\[
		k\bar{d}p_d(S^{-a}T^{-1}m_g)(S^{-a}T^{-1}p_g)=d(S^{-a}T^{-1}g)=0
	\]
	Thus $p_d(S^{-a}T^{-1}m_g)=0$ so there exists $\gamma\colon I(h)\rightarrow I(d)$ such that $\gamma p_h=p_d$. Also, $(S^aTd)g=0$, so there is $\bar{g}\colon D\rightarrow S^aTK(d)$ such that $g=k\bar{g}$. Finally, we have
	\[
		k\bar{g}m_hp_h=gh=d=k\bar{d}p_d=k\bar{d}\gamma p_h
	\]
	So $\bar{g}m_h=\bar{d}\gamma$, which gives $g'\colon D'\rightarrow E'$ such that the following diagram commutes:
	\[
		\begin{tikzpicture}
			\node(A) at (0,0){$I(h)$};
			\node(B) at (3,0){$D$};
			\node(C) at (6,0){$D'$};
			\node(D) at (0,-2){$I(d)$};
			\node(E) at (3,-2){$S^aTK(d)$};
			\node(F) at (6,-2){$E'$};

			\node(G) at (5.5,-1.5){$\pushout$};

			\draw[\normalmono] (A)--(B) node[above,midway] {$m_h$};
			\draw[\normalepi] (B)--(C) node[above,midway] {$S^{-1}p_f$};
			\draw[\normalepi] (A)--(D) node[left,midway] {$\gamma$};
			\draw[\map] (B)--(E) node[left,midway] {$\bar{g}$};
			\draw[\map] (C)--(F) node[right,midway] {$g'$};
			\draw[\normalmono] (D)--(E) node[below,midway] {$\bar{d}$};
			\draw[\normalepi] (E)--(F) node[below,midway] {$q$};
		\end{tikzpicture}
	\]
	The right-hand square is a pushout by Lemma \ref{lemma pushout recognition}. Let us now show exactness. $f'=(S^{-1}p_f)m_f$, hence it is normal by Lemma \ref{lemma diexact}. $\ker(g)=\ker(k\bar{g})=\ker(\bar{g})$, so $I(\bar{g})=I(g)=K(S^aTh)$. Thus $\bar{g}$ is normal, so it can be written as $\bar{g}=\xi p_g$ with $\xi\colon I(g)\normono S^aTK(d)$. Since pushouts along normal epimorphisms preserve normal maps \cite[Proposition 5.3.3]{PVdL3}, we obtain a commutative diagram
	\[
		\begin{tikzpicture}
			\node(A) at (0,0){$D$};
			\node(B) at (3,0){$I(g)$};
			\node(C) at (6,0){$S^aTK(d)$};
			\node(D) at (0,-2){$D'$};
			\node(E) at (3,-2){$I(g')$};
			\node(F) at (6,-2){$E'$};

			\draw[\normalepi] (A)--(B) node[above,midway] {$p_g$};
			\draw[\normalmono] (B)--(C) node[above,midway] {$\xi$};
			\draw[\normalepi] (A)--(D) node[left,midway] {$S^{-1}p_f$};
			\draw[\normalepi] (B)--(E) node[left,midway] {$e$};
			\draw[\normalepi] (C)--(F) node[right,midway] {$q$};
			\draw[\normalepi] (D)--(E);
			\draw[\normalmono] (E)--(F);
			\draw[\map] (A)..controls(2,1) and (4,1)..(C) node[above,midway] {$\bar{g}$};
			\draw[\map] (D)..controls(2,-3) and (4,-3)..(F) node[below,midway] {$g'$};
		\end{tikzpicture}
	\]
	which proves that $g'$ is normal. $g'$ is a pushout of $\bar{g}$, so they have the same cokernel by Lemma \ref{lemma cokernel of pushout}. Thus the right-hand square is a pullback by Lemma \ref{lemma pullback recognition}, so $I(d)$, which is the kernel of $q$, is also the kernel of $e$. Since $\gamma\colon I(h)\rightarrow I(d)$ is an epimorphism, this means that
	\[
		I(g')=C\big(I(d)\rightarrow I(g)\big)=C\big(I(h)\rightarrow I(g)\big)=C(p_gm_h)
	\]
	so $e$ is a pushout of $S^{-1}p_f$ along $p_g$ by Lemma \ref{lemma kernel of composite}. The same Lemma \ref{lemma kernel of composite} then gives
	\[
		I(g')=C\big((S^{-1}p_f)m_f\big)=C(f')
	\]
	since $p_g$ is a cokernel of $m_f$. $\bar{h}$ is normal because it is a pullback of $h$, and $I(h')=I\big(h'(S^{-(a+1)}T^{-1}q)\big)=I(S^{-1}\bar{h}\big)$. Moreover, the diagram
	\[
		\begin{tikzpicture}
			\node(A) at (0,0){$I(\bar{h})$};
			\node(B) at (3,0){$SD'$};
			\node(C) at (6,0){$D'$};
			\node(D) at (0,-2){$I(h)$};
			\node(E) at (3,-2){$D$};
			\node(F) at (6,-2){$D'$};

			\node(G) at (0.5,-0.5){$\pullback$};

			\draw[\normalmono] (A)--(B);
			\draw[\map] (B)--(C) node[above,midway] {$f'$};
			\draw[\normalmono] (A)--(D);
			\draw[\normalmono] (B)--(E) node[left,midway] {$m_f$};
			\draw[double,double distance=1mm] (C)--(F);
			\draw[\normalmono] (D)--(E) node[below,midway] {$m_h$};
			\draw[\normalepi] (E)--(F) node[below,midway] {$S^{-1}p_f$};
		\end{tikzpicture}
	\]
	proves that $I(h')=S^{-1}I(\bar{h})=S^{-1}K(f')$ by Lemma \ref{lemma kernel of composite}. Finally, since $k\xi p_g=k\bar{g}=g=m_gp_g$, $\xi$ is the unique morphism such that $k\xi=m_g$. The following diagram
	\[
		\begin{tikzpicture}
			\node(A) at (0,0){$S^aTK(h)$};
			\node(B) at (3,0){$S^aTK(d)$};
			\node(C) at (6,0){$S^aTI(\bar{h})$};
			\node(D) at (0,-2){$S^aTK(h)$};
			\node(E) at (3,-2){$E$};
			\node(F) at (6,-2){$S^aTI(h)$};

			\node(G) at (3.5,-0.5){$\pullback$};

			\draw[\normalmono] (A)--(B) node[above,midway] {$\xi$};
			\draw[\normalepi] (B)--(C);
			\draw[double,double distance=1mm] (A)--(D);
			\draw[\normalmono] (B)--(E) node[left,midway] {$k$};
			\draw[\normalmono] (C)--(F);
			\draw[\normalmono] (D)--(E) node[below,midway] {$m_g$};
			\draw[\normalepi] (E)--(F) node[below,midway] {$S^aTp_h$};
		\end{tikzpicture}
	\]
	implies that $C(\xi)=S^aTI(\bar{h})$. Hence, we find that $C(g')=S^aTI(\bar{h})=S^{a+1}TI(h')$ with the following diagram in which the left-hand square is both a pullback and a pushout by Lemma \ref{lemma pb with normal epis}:
	\[
		\begin{tikzpicture}
			\node(A) at (0,0){$S^aTK(h)$};
			\node(B) at (3,0){$S^aTK(d)$};
			\node(C) at (6,0){$S^aTI(\bar{h})$};
			\node(D) at (0,-2){$I(g')$};
			\node(E) at (3,-2){$E'$};
			\node(F) at (6,-2){$S^aTI(\bar{h})$};

			\node(G) at (2.5,-1.5){$\pushout$};

			\draw[\normalmono] (A)--(B) node[above,midway] {$\xi$};
			\draw[\normalepi] (B)--(C);
			\draw[\normalepi] (A)--(D) node[left,midway] {$e$};
			\draw[\normalepi] (B)--(E) node[right,midway] {$q$};
			\draw[double,double distance=1mm] (C)--(F);
			\draw[\normalmono] (D)--(E);
			\draw[\normalepi] (E)--(F);
		\end{tikzpicture}
	\]
	We then have an exact sequence
	\[
		\cdots\rightarrow S^{-(a+1)}T^{-1}D'\xrightarrow{S^{-(a+1)}T^{-1}g'} S^{-(a+1)}T^{-1}E'\xrightarrow{h'}D'\xrightarrow{S^{-1}f'} S^{-1}D'\xrightarrow{S^{-1}g'}\cdots
	\]
	as required.
\end{proof}

Let $\E$ be an exact couple of degree $a$. The previous theorem allows us to construct a sequence of exact couples $(\E^r)_{r\geqslant a}$, with $\E^r=(D^r,E^r,f^r,g^r,h^r)$, by setting $\E^a=\E$ and $\E^{r+1}=(\E^r)'$, whenever each morphism $d^r=g^rh^r$ is normal. In this case, we call $\E$ a \Def{normal} exact couple. We obtain a spectral sequence $(E^r,d^r)_{r\geqslant a}$ whose differential is the given map $d^r$.

\begin{lemma}\label{lemma Zr}
	Let $\E=(D,E,f,g,h)$ be a normal exact couple of degree $a$, and $(E,d)$ the spectral sequence that arises from $\E$. Then we have the formula
	\[
		S^{-a}T^{-1}Z^r=h^{-1}I\Big(f(Sf)\cdots (S^{r-a-1}f)\Big)
	\]
\end{lemma}
\begin{proof}
	First of all, we show by induction on $r\geqslant a$ that the formula
	\[
		I\Big(f(Sf)\cdots(S^{r-a}f)\Big)=S^{r-a}I(f^r)
	\]
	holds. Since this is obvious for $r=a$, assume it holds for $r\geqslant a$. By construction of $f'$ we have:
	\begin{align*}
		I\Big(f(Sf)\cdots(S^{r+1-a}f)\Big) & =I\Big(m_fp_f(Sm_f)(Sp_f)\cdots(S^{r+1-a}m_f)(S^{r+1-a}p_f)\Big) \\
		                                   & =I\Big(p_f(Sm_f)(Sp_f)\cdots(S^{r+1-a}m_f)\Big)                  \\
		                                   & =I\Big((Sf')\cdots(S^{r+1-a}f')\Big)                             \\
		                                   & =SI\Big(f'(Sf')\cdots(S^{r-a}f')\Big)                            \\
		                                   & =S^{r+1-a}I(f^{r+1})
	\end{align*}
	The second equality holds because post-composition with the monomorphism $m_f$ does not change the image, as well as pre-composition with the epimorphism $p_f$. The last equality holds by the induction hypothesis applied on $f'$, since $f^{r+1}=(f')^r$. Now, we prove by induction on $r\geqslant a$ that there is a pullback diagram
	\begin{equation}\label{diag lemma formula Zr induction hyp}
		\begin{tikzpicture}[baseline=(G)]
			\node(A) at (0,0){$Z^r$};
			\node(B) at (3,0){$E$};
			\node(C) at (0,-1.5){$E^r$};
			\node(D) at (0,-3){$S^rTD^r$};
			\node(E) at (3,-3){$S^aTD$};

			\node(F) at (0.5,-0.5){$\pullback$};

			\node(G) at (-1,-1.65){};

			\draw[\normalmono] (A)--(B);
			\draw[\mono] (D)--(E);
			\draw[\normalepi] (A)--(C);
			\draw[\map] (C)--(D) node[left,midway] {$S^rTh^r$};
			\draw[\map] (B)--(E) node[right,midway] {$S^aTh$};
		\end{tikzpicture}
	\end{equation}

	The horizontal maps are isomorphisms when $r=a$, so then it is pullback. Assume it is a pullback for some $r\geqslant a$. By construction, there is a pullback
	\[
		\begin{tikzpicture}
			\node(A) at (0,0){$K(d^r)$};
			\node(B) at (3,0){$S^{-r}T^{-1}E^r$};
			\node(C) at (0,-1.5){$S^{-r}T^{-1}E^{r+1}$};
			\node(D) at (0,-3){$SD^{r+1}$};
			\node(E) at (3,-3){$D^r$};

			\node(F) at (0.5,-0.5){$\pullback$};

			\draw[\normalmono] (A)--(B);
			\draw[\normalmono] (D)--(E);
			\draw[\normalepi] (A)--(C);
			\draw[\map] (C)--(D) node[left,midway] {$Sh^{r+1}$};
			\draw[\map] (B)--(E) node[right,midway] {$h^r$};
		\end{tikzpicture}
	\]
	Applying $S^rT$ to this diagram, and pasting it to Diagram \ref{diag lemma formula Zr induction hyp} and the right-hand square of Diagram \ref{diag Br Zr} that defines $Z^{r+1}$, we obtain the result for $r+1$. However, we have for $r\geqslant a$ that
	\[
		S^rTD^r=S^{r-1}TI(f^{r-1})=S^aTI\Big(f(Sf)\cdots (S^{r-a-1}f)\Big)
	\]
	whence the result follows.
\end{proof}

\begin{lemma}\label{lemma Br}
	Let $\E=(D,E,f,g,h)$ be a normal exact couple of degree $a$, and $(E,d)$ the spectral sequence that arises from $\E$. Then we have the formula
	\[
		B^r=gK\Big((S^{a-r}f)\cdots(S^{-2}f)(S^{-1}f)\Big)
	\]
\end{lemma}
\begin{proof}
	We often use here that $SI(h)=K(f)$. We prove first by induction on $r\geqslant a$ that we have a pullback diagram
	\[
		\begin{tikzpicture}
			\node(A) at (0,0){$P^r_f$};
			\node(B) at (3,0){$D$};
			\node(C) at (0,-2){$I(h^r)$};
			\node(D) at (3,-2){$D^r$};

			\node(E) at (0.5,-0.5){$\pullback$};

			\draw[\normalmono] (A)--(B);
			\draw[\normalmono] (C)--(D);
			\draw[\normalepi] (A)--(C);
			\draw[\normalepi] (B)--(D);
		\end{tikzpicture}
	\]
	where the right-hand side morphism is the composite $D=D^a\norepi D^{a+1}\norepi\cdots\norepi D^r$ and $P^r_f\coloneq K\big((S^{a-r-1}f)\cdots(S^{-2}f)(S^{-1}f)\big)$. This holds for $r=a$, so suppose that it holds for $r\geqslant a$.
	We have:
	\begin{align*}
		P^{r+1}_f & =K\Big((S^{a-(r+1)-1}f)\cdots(S^{-2}f)(S^{-1}f)\Big)                                    \\
		          & =K\Big((S^{a-(r+1)-1}m_f)(S^{a-(r+1)-1}p_f)\cdots(S^{-2}p_f)(S^{-1}m_f)(S^{-1}p_f)\Big) \\
		          & =(S^{-1}p_f)^{-1}K\Big((S^{a-(r+1)-1}p_f)\cdots(S^{-2}p_f)(S^{-1}m_f)\Big)              \\
		          & =(S^{-1}p_f)^{-1}K\Big((S^{a-r-1}f')\cdots(S^{-1}f')\Big)                               \\
		          & =(S^{-1}p_f)^{-1}P^r_{f'}
	\end{align*}
	Hence, using the hypothesis on $f'$, we have a composite of pullback squares
	\[
		\begin{tikzpicture}
			\node(A) at (0,0){$P^{r+1}_f$};
			\node(B) at (3,0){$D$};
			\node(C) at (0,-2){$P^r_{f'}$};
			\node(D) at (3,-2){$D'$};
			\node(E) at (0,-4){$I(h^{r+1})$};
			\node(F) at (3,-4){$D^{r+1}$};

			\node(G) at (0.5,-0.5){$\pullback$};
			\node(H) at (0.5,-2.5){$\pullback$};

			\draw[\normalmono] (A)--(B);
			\draw[\normalmono] (C)--(D);
			\draw[\normalmono] (E)--(F);
			\draw[\normalepi] (A)--(C);
			\draw[\normalepi] (B)--(D) node[right, midway] {$S^{-1}p_f$};
			\draw[\normalepi] (C)--(E);
			\draw[\normalepi] (D)--(F);
		\end{tikzpicture}
	\]
	which concludes the induction. Now let us show by induction on $r\geqslant a$ that there are two composable pushout squares
	\begin{equation}\label{diag lemma formula Br induction hyp}
		\begin{tikzpicture}[baseline=(I)]
			\node(A) at (0,0){$D$};
			\node(B) at (3,0){$I(g)$};
			\node(C) at (6,0){$Z^r$};
			\node(D) at (0,-2){$D^r$};
			\node(E) at (3,-2){$I(g^r)$};
			\node(F) at (6,-2){$E^r$};

			\node(G) at (2.5,-1.5){$\pushout$};
			\node(H) at (5.5,-1.5){$\pushout$};

			\node(I) at (-1,-1.2){};

			\draw[\normalepi] (A)--(B);
			\draw[\normalmono] (B)--(C);
			\draw[\normalepi] (A)--(D);
			\draw[\normalepi] (B)--(E);
			\draw[\normalepi] (C)--(F);
			\draw[\normalepi] (D)--(E);
			\draw[\normalmono] (E)--(F);
		\end{tikzpicture}
	\end{equation}
	For $r=a$, this is just the normal factorization of $g$ on both rows. Assume this holds for some $r\geqslant a$. Using the constructions of $Z^{r+1}$ and $g^{r+1}$, we end up with a diagram
	\[
		\begin{tikzpicture}
			\node(A) at (0,0){$D$};
			\node(B) at (3,0){$I(g)$};
			\node(C) at (6,0){$Z^{r+1}$};
			\node(D) at (9,0){$Z^r$};
			\node(E) at (0,-2){$D^r$};
			\node(F) at (3,-2){$I(g^r)$};
			\node(G) at (6,-2){$S^rTK(d^r)$};
			\node(H) at (9,-2){$E^r$};
			\node(I) at (0,-4){$D^{r+1}$};
			\node(J) at (3,-4){$I(g^{r+1})$};
			\node(K) at (6,-4){$E^{r+1}$};

			\node(L) at (1.5,-1){$(1)$};
			\node(M) at (4.5,-1){$(2)$};
			\node(N) at (7.5,-1){$(3)$};
			\node(O) at (1.5,-3){$(4)$};
			\node(P) at (4.5,-3){$(5)$};

			\draw[\normalepi] (A)--(B);
			\draw[\normalmono] (B)--(C);
			\draw[\normalmono] (C)--(D);
			\draw[\normalepi] (A)--(E);
			\draw[\normalepi] (B)--(F);
			\draw[\normalepi] (C)--(G);
			\draw[\normalepi] (D)--(H);
			\draw[\normalepi] (E)--(F);
			\draw[\normalmono] (F)--(G);
			\draw[\normalmono] (G)--(H);
			\draw[\normalepi] (E)--(I);
			\draw[\normalepi] (F)--(J);
			\draw[\normalepi] (G)--(K);
			\draw[\normalepi] (I)--(J);
			\draw[\normalmono] (J)--(K);
		\end{tikzpicture}
	\]
	in which:
	\begin{itemize}
		\item $(1)$, $(2)+(3)$ and $(1)+(2)+(3)$ are pushout rectangles by the induction hypothesis;
		\item $(4)$, $(5)$ and $(4)+(5)$ are pushout rectangles by the construction of $g^{r+1}$;
		\item $(2)+(3)$ is a pullback rectangle because it is the right-hand square of Diagram \ref{diag lemma formula Br induction hyp};
		\item $(3)$ is a pullback square by the definition of $Z^{r+1}$;
	\end{itemize}
	Therefore, $(2)$ is a pullback square, and thus a pushout square by Lemma \ref{lemma pb with normal epis}, so $(1)+(4)$ and $(2)+(5)$ are pushout rectangles as wanted. Now, consider the commutative cube
	\[
		\begin{tikzpicture}
			\node(A) at (-1.5,-1.5){$P^r_f$};
			\node(B) at (1.5,-1.5){$D$};
			\node(C) at (-1.5,-4){$I(h^r)$};
			\node(D) at (1.5,-4){$D^r$};
			\node(E) at (0,0){$gP^r_f$};
			\node(F) at (3,0){$I(g)$};
			\node(G) at (0,-2.5){$I(d^r)$};
			\node(H) at (3,-2.5){$I(g^r)$};

			\node(I) at (0,-1.5){};
			\node(J) at (1.5,-2.5){};

			\draw[\normalmono] (A)--(B);
			\draw[\normalmono] (C)--(D);
			\draw[\normalepi] (A)--(C);
			\draw[\normalepi] (B)--(D);
			\draw[\normalepi] (A)--(E);
			\draw[\normalepi] (B)--(F);
			\draw[\normalepi] (C)--(G);
			\draw[\normalepi] (D)--(H);
			\draw[\normalmono] (E)--(F);
			\draw[\normaldemimono] (G)--(J);
			\draw[\map] (J)--(H);
			\draw[-] (E)--(I);
			\draw[\normalepi] (I)--(G);
			\draw[\normalepi] (F)--(H);
		\end{tikzpicture}
	\]
	in which the front face is both a pullback and a pushout square by Lemma \ref{lemma pb with normal epis}, and the right-hand side face is a pushout square. Hence the composite of the front and right-hand side faces, which is also the composite of the left-hand side and the back faces, is a pushout square. This implies together with $I(h^r)\norepi I(d^r)$ being an epimorphism that the back face of the cube is a pushout square by Lemma \ref{lemma pushout with epi}, so it is also a pullback square by Lemma \ref{lemma pb with normal epis}. By composition with the right-hand square of Diagram \ref{diag lemma formula Br induction hyp}, this means that $gP^r_f\cong B^{r+1}$.
\end{proof}

\section{Simplicial Objects}

\begin{definition}\label{def simplicial object}
	A \Def{simplicial object} $G$ in $\C$ is a functor from $\deltacat^\op$ to $\C$, where $\deltacat$ is the skeletal category of finite totally ordered sets with non-decreasing maps. This means a graded set of objects $G_q$ of $\C$ for $q\geqslant0$, together with maps $d_i\colon G_q\rightarrow G_{q-1}$ and $s_i\colon G_q\rightarrow G_{q+1}$ for $0\leqslant i\leqslant q$, called respectively \Def{faces} and \Def{degeneracies}, that satisfy the following identities:
	\begin{align*}
		d_id_j & =d_{j-1}d_i & \text{ if } & i<j                 \\
		s_is_j & =s_{j+1}s_i & \text{ if } & i\geqslant j        \\
		d_is_j & =1          & \text{ if } & i=j\text{ or }i=j+1 \\
		d_is_j & =s_{j-1}d_i & \text{ if } & i<j                 \\
		d_is_j & =s_jd_{i-1} & \text{ if } & i>j+1
	\end{align*}
\end{definition}

\begin{definition}\label{def Moore complex}
	Let $G$ be a simplicial object of $\C$. We define a chain complex $(NG,\partial)$ by setting $(NG)_0=G_0$,
	\[
		(NG)_q\coloneq \bigcap_{i=0}^{q-1}K(d_i\colon G_q\rightarrow G_{q-1})\leqslant G_q
	\]
	and $\partial_q\colon (NG)_q\rightarrow (NG)_{q-1}$ induced by the face $d_q$, for $q\geqslant1$, and $(NG)_q=0$ for $q<0$. The identity $d_id_q=d_{q-1}d_i$ ensures that it is well defined, and that $\partial^2=0$. $(NG,\partial)$ is called the \Def{Moore complex} of $G$. It is a normal chain complex \cite[Theorem 3.6]{EveraertVdLinden}, meaning that $\partial_q$ is a normal map for each $q$. The assignment $G\mapsto NG$ defines a functor $\simp(\C)\rightarrow\ch_{\geqslant0}(\C)$ which is exact \cite[Proposition 5.6]{EveraertVdLinden}. For each $q\geqslant0$, we write $H_qG$ or $H_q(G)$ for the \Def{$q$-th homology object} of $G$, which is by definition $C\big(I(\partial_{q+1})\rightarrow K(\partial_q)\big)$, the $q$-th homology object of $NG$.
\end{definition}

Let $G$ be a simplicial object in $\C$. We define another simplicial object $EG$ by setting $(EG)_q\coloneq K(d_0^{q+1}\colon G_{q+1}\rightarrow G_0)$; its faces and degeneracies $d_i\colon (EG)_q\rightarrow (EG)_{q-1}$ and $s_i\colon (EG)_q\rightarrow (EG)_{q+1}$ are induced respectively by $d_i\colon G_{q+1}\rightarrow G_q$ and $s_i\colon G_{q+1}\rightarrow G_{q+2}$ for each $0\leqslant i\leqslant q$.

\begin{proposition}\label{prop homology EG}
	For each $q\geqslant0$, $H_q(EG)=0$.
\end{proposition}
\begin{proof}
	The complex $NEG$ is $(NEG)_0=(EG)_0=K(d_0\colon G_1\rightarrow G_0)$, and for each $q\geqslant1$:
	\begin{align*}
		(NEG)_q & =\bigcap_{i=0}^{q-1}K\big(d_i\colon(EG)_q\rightarrow(EG)_{q-1}\big)                                                    \\
		        & =\bigcap_{i=0}^{q-1}K\big(d_i\colon K(d_0^{q+1})\rightarrow K(d_0^q)\big)                                              \\
		        & =K(d_0^{q+1})\cap\bigcap_{i=0}^{q-1}K(d_i\colon G_{q+1}\rightarrow G_q)                                                \\
		        & =\bigcap_{i=0}^{q-1}K(d_i\colon G_{q+1}\rightarrow G_q)                   & \text{ since }K(d_0)\leqslant K(d_0^{q+1})
	\end{align*}
	and the differential $\partial_q\colon(NEG)_q\rightarrow(NEG)_{q-1}$ is induced by $d_q$. Let $q\geqslant0$. The morphism $s_{q+1}$ restricts to a map $(NEG)_q\rightarrow G_{q+2}$, which itself induces $\sigma_{q+1}\colon K(\partial_q)\rightarrow(NEG)_{q+1}$ because the identity $d_is_{q+1}=s_qd_i$ holds for every $0\leqslant i\leqslant q$ and
	\[
		K(\partial_q)=(NEG)_q\cap K(d_q)=\bigcap_{i=0}^qK(d_i\colon G_{q+1}\rightarrow G_q)
	\]
	The lemma follows, since $\sigma_{q+1}$ is a splitting of the map $(NEG)_{q+1}\rightarrow K(\partial_q)$ induced by $\partial_{q+1}$, thus proving that $NEG$ is exact.
\end{proof}

For $i\leqslant q+1$, we have $d_id_{q+2}=d_{q+1}d_i$, and for $i\leqslant q$, we have $s_id_{q+1}=d_{q+2}s_i$. Therefore there is a morphism $\theta\colon EG\rightarrow G$ whose $q$-th component $\theta_q$ is the restriction of $d_{q+1}\colon G_{q+1}\norepi G_q$ to $(EG)_q$. The morphism $\theta$ is normal by Lemma~\ref{lemma diexact}.

\begin{lemma}
	Let $CH_0G$ be the constant simplicial object with value $H_0G$, let $\pi\colon G_0\norepi H_0G$ be the canonical projection, and for each $q\geqslant0$, define $\rho_q\colon G_q\rightarrow H_0G$ as $\rho_q=\pi d_0^q$. Then $\rho\colon G\rightarrow CH_0G$ is a morphism of simplicial objects, and is a cokernel of $\theta$.
\end{lemma}
\begin{proof}
	$H_0G$ is the cokernel of the differential $\partial_1\colon K(d_0)\rightarrow G_0$ that comes from the Moore complex $NG$ and that is induced by $d_1\colon G_1\rightarrow G_0$. Since $d_1$ and $d_0$ are both split by $s_0$, the morphism $\pi$ is a coequalizer of $d_0$ and $d_1$. Fix $q\geqslant 0$. For $1\leqslant i\leqslant q$, the identities $d_0d_i=d_{i-1}d_0$, $d_0s_i=s_{i-1}d_0$ and $d_0s_0=1$ hold. This implies that $d_0^{q-1}d_q=d_1d_0^{q-1}$ and $d_0^qs_q=s_0d_0^q$, and for each $0\leqslant i\leqslant q-1$, $d_0^{q-1}d_i=d_0^q$ and $d_0^qs_i=d_0^{q-1}$. Therefore, since $\pi d_0=\pi d_1$ and $d_1s_1=1$, we have for each $0\leqslant i\leqslant q$
	\begin{align*}
		\rho_{q-1}d_i & =\pi d_0^{q-1}d_i=\pi d_0^q=\rho_q     \\
		\rho_qs_i     & =\pi d_0^qs_i=\pi d_0^{q-1}=\rho_{q-1}
	\end{align*}
	Thus $\rho$ is a morphism. Let $k\colon (EG)_q\normono G_{q+1}$ be the kernel of $d_0^{q+1}$, so that we have $\theta_q=d_{q+1}k$.
	Then
	\[
		\rho_q\theta_q=\pi d_0^qd_{q+1}k=\pi d_1d_0^qk=\pi d_0^{q+1}k=0
	\]
	Let $f\colon G_q\rightarrow X$ in $\C$ be such that $f\theta_q=fd_{q+1}k=0$. As $d_0^{q+1}$ is a cokernel of $k$, there exists a unique $g\colon G_0\rightarrow X$ such that $fd_{q+1}=gd_0^{q+1}$.
	Thus
	\[
		g=gd_0^{q+1}s_0^{q+1}=fd_{q+1}s_0^{q+1}=fs_0^q
	\]
	and
	\[
		gd_1=fs_0^qd_1=fd_{q+1}s_0^q=gd_0^{q+1}s_0^q=gd_0
	\]
	So $g$ coequalizes $d_0$ and $d_1$. Hence there exists $h\colon H_0G\rightarrow X$ such that $g=h\pi$, and thus
	\[
		f=fd_{q+1}s_q=gd_0^{q+1}s_q=gd_0^q=h\pi d_0^q=h\rho_q
	\]
	The uniqueness of $h$ follows from $\rho_q$ being an epimorphism.
\end{proof}

We denote $\kappa\colon \Omega G\normono G$ the kernel of $\theta$, and we have in $\simp(\C)$ an exact sequence
\begin{equation}\label{exact sequence simp objects}
	\begin{tikzpicture}[baseline=(G)]
		\node(A) at (0,0){$0$};
		\node(B) at (2,0){$\Omega G$};
		\node(C) at (4,0){$EG$};
		\node(D) at (6,0){$G$};
		\node(E) at (8,0){$CH_0G$};
		\node(F) at (10,0){$0$};

		\node(G) at (-1,-0.1){};

		\draw[\map] (A)--(B);
		\draw[\map] (B)--(C) node[above, midway] {$\kappa$};
		\draw[\map] (C)--(D) node[above, midway] {$\theta$};
		\draw[\map] (D)--(E) node[above, midway] {$\rho$};
		\draw[\map] (E)--(F);
	\end{tikzpicture}
\end{equation}
Writing $IG$ for the image of $\theta$, we obtain two short exact sequences
\[
	\begin{tikzpicture}
		\node(A) at (0,0){$\Omega G$};
		\node(B) at (2,0){$EG$};
		\node(C) at (4,0){$IG$};

		\draw[\normalmono] (A)--(B) node[above, midway] {$\kappa$};
		\draw[\normalepi] (B)--(C);

		\node(D) at (0.05,-1){$IG$};
		\node(E) at (2,-1){$G$};
		\node(F) at (4.3,-1){$CH_0G$};

		\draw[\normalmono] (D)--(E);
		\draw[\normalepi] (E)--(F) node[above, midway] {$\rho$};
	\end{tikzpicture}
\]
which give two long exact sequences \cite[Proposition 2.4]{EveraertVdLinden}
\[
	\begin{tikzpicture}
		\node(A) at (0,0){$\cdots$};
		\node(B) at (2.3,0){$H_p(EG)$};
		\node(C) at (4.6,0){$H_p(IG)$};
		\node(D) at (6.9,0){$H_{p-1}(\Omega G)$};
		\node(E) at (9.2,0){$H_{p-1}(EG)$};
		\node(F) at (11.5,0){$\cdots$};

		\draw[\map] (A)--(B);
		\draw[\map] (B)--(C);
		\draw[\map] (C)--(D);
		\draw[\map] (D)--(E);
		\draw[\map] (E)--(F);

		\node(A') at (0,-1){$\cdots$};
		\node(B') at (2.3,-1){$H_{p+1}(CH_0G)$};
		\node(C') at (4.6,-1){$H_p(IG)$};
		\node(D') at (6.9,-1){$H_p(G)$};
		\node(E') at (9.2,-1){$H_p(CH_0G)$};
		\node(F') at (11.5,-1){$\cdots$};

		\draw[\map] (A')--(B');
		\draw[\map] (B')--(C');
		\draw[\map] (C')--(D');
		\draw[\map] (D')--(E');
		\draw[\map] (E')--(F');
	\end{tikzpicture}
\]
For each $p\geqslant 1$, $H_p(CH_0G)=0$, and so we established with Proposition \ref{prop homology EG} the isomorphism:
\begin{equation}\label{iso omega}
	H_{p-1}(\Omega G)\cong H_p(G)
\end{equation}

\begin{definition}
	A \Def{double simplicial object} $G$ in $\C$ is a functor from $(\deltacat\times\deltacat)^\op$ to $\C$. This is an indexed set of objects $G_{p,q}$ in $\C$ with $p,q\geqslant0$, together with vertical and horizontal face maps and degeneracy maps as follows:
	\begin{align*}
		d_i^v & \colon G_{p,q}\rightarrow G_{p,q-1} \\
		s_i^v & \colon G_{p,q}\rightarrow G_{p,q+1} \\
		d_i^h & \colon G_{p,q}\rightarrow G_{p-1,q} \\
		s_i^h & \colon G_{p,q}\rightarrow G_{p+1,q}
	\end{align*}
	The $d_i^v$ and $s_i^v$ commute with the $d_i^h$ and the $s_i^h$, and we find the same identities as in Definition \ref{def simplicial object} between vertical maps and between horizontal maps.
\end{definition}

We write $\dsimp(\C)$ for the category of double simplicial object in $\C$.
We can view a double simplicial object $G$ as a simplicial object in $\simp(\C)$. Indeed, for each $p\geqslant0$, $G^v_p\coloneq(G_{p,\ast},d_i^v,s_i^v)$ is a simplicial object in $\C$, \ie an object of $\simp(\C)$, and the families of horizontal face maps $(d_i^h\colon G_{p,q}\rightarrow G_{p-1,q})_q$ and degeneracy maps $(s_i^h\colon G_{p,q}\rightarrow G_{p+1,q})_q$ are morphisms in $\simp(\C)$ and make $G^v$ a simplicial object.

\begin{definition}
	The \Def{$q$-th vertical homology} of a double simplicial object $G$ is the simplicial object $H_q^vG$ defined by $(H_q^vG)_p\coloneq H_q(G_{p,\ast})$, and whose face maps and degeneracy maps are induced in homology respectively by $d_i^h\colon G_{p,q}\rightarrow G_{p-1,q}$ and $s_i^h\colon G_{p,q}\rightarrow G_{p+1,q}$. We write $H_p^hH_q^vG$ for the $p$-th homology object of $H_q^vG$.
\end{definition}

Let $G$ be a double simplicial object. Let us apply \ref{exact sequence simp objects} to $G$ seen as the simplicial object $G^v$:
\[
	\begin{tikzpicture}
		\node(A) at (0,0){$0$};
		\node(B) at (2,0){$\Omega_v G$};
		\node(C) at (4,0){$E_vG$};
		\node(D) at (6,0){$G$};
		\node(E) at (8,0){$C_vH_0^vG$};
		\node(F) at (10,0){$0$};

		\draw[\map] (A)--(B);
		\draw[\map] (B)--(C) node[above, midway] {$\kappa_v$};
		\draw[\map] (C)--(D) node[above, midway] {$\theta_v$};
		\draw[\map] (D)--(E) node[above, midway] {$\rho_v$};
		\draw[\map] (E)--(F);
	\end{tikzpicture}
\]

\begin{definition}
	Given a double simplicial object $G$, we define its \Def{diagonal simplicial object} $\Delta G$ by setting
	\begin{align*}
		(\Delta G)_n=G_{n,n} &  & d_i=d_i^vd_i^h &  & s_i=s_i^vs_i^h
	\end{align*}
	The assignment $G\mapsto\Delta G$ is an exact functor $\Delta\colon\dsimp(\C)\rightarrow\simp(\C)$.
\end{definition}

If $I_vG$ denotes the image of $\theta_v$, then by exactness of $\Delta$ we obtain the long exact sequences
\[
	\begin{tikzpicture}
		\node(A) at (0,0){$\cdots$};
		\node(B) at (1.9,0){$H_p(\Delta E_vG)$};
		\node(C) at (4.3,0){$H_p(\Delta I_vG)$};
		\node(D) at (7,0){$H_{p-1}(\Delta \Omega_v G)$};
		\node(E) at (9.9,0){$H_{p-1}(\Delta E_vG)$};
		\node(F) at (11.9,0){$\cdots$};

		\draw[\map] (A)--(B);
		\draw[\map] (B)--(C);
		\draw[\map] (C)--(D);
		\draw[\map] (D)--(E);
		\draw[\map] (E)--(F);

		\node(A') at (0,-1){$\cdots$};
		\node(B') at (1.9,-1){$H_p(\Delta I_vG)$};
		\node(C') at (4.3,-1){$H_p(\Delta G)$};
		\node(D') at (6.9,-1){$H_p(\Delta C_vH_0^vG)$};
		\node(E') at (9.9,-1){$H_{p-1}(\Delta I_vG)$};
		\node(F') at (11.9,-1){$\cdots$};

		\draw[\map] (A')--(B');
		\draw[\map] (B')--(C');
		\draw[\map] (C')--(D');
		\draw[\map] (D')--(E');
		\draw[\map] (E')--(F');
	\end{tikzpicture}
\]

\begin{proposition}
	For each $p\geqslant0$, $H_p(\Delta E_vG)=0$.
\end{proposition}
\begin{proof}
	We adapt the proof of \ref{prop homology EG}. The complex $(E_vG)$ is defined by $(E_vG)_{p,q}=K\big((d_0^h)^{p+1}\colon G_{p+1,q}\rightarrow G_{0,q}\big)$, and the face maps $d_i^v$, $d_i^h$ and degeneracy maps $s_i^v$, $s_i^h$ are induced respectively by the face maps  $d_i^v$, $d_i^h$ and degeneracy maps $s_i^v$, $s_i^h$ of the original complex $G$. Therefore the complex $(N\Delta E_vG)$ is defined by $(N\Delta E_vG)_0=(E_vG)_{0,0}=K(d_0^h\colon G_{1,0}\rightarrow G_{0,0})$, and for each $p\geqslant1$:
	\begin{align*}
		(N\Delta E_vG)_p & =\bigcap_{i=0}^{p-1}K\big(d_i^hd_i^v\colon(E_vG)_{p,p}\rightarrow(E_vG)_{p-1,p-1}\big)                                                        \\
		                 & =\bigcap_{i=0}^{p-1}K\big(d_i^hd_i^v\colon K((d_0^h)^{p+1})\rightarrow K((d_0^h)^p)\big)                                                      \\
		                 & =K\big((d_0^h)^{p+1}\colon G_{p+1,p}\rightarrow G_{0,p}\big)\cap\bigcap_{i=0}^{p-1}K\big(d_i^hd_i^v\colon G_{p+1,p}\rightarrow G_{p,p-1}\big) \\
	\end{align*}
	and the differential $\partial_p\colon(N\Delta E_vG)_p\rightarrow(N\Delta E_vG)_{p-1}$ is induced by $d_p^hd_p^v$. Let $p\geqslant0$. $s_{p+1}^hs_{p+1}^v$ restricts to a map $(N\Delta E_vG)_p\rightarrow G_{p+2,p+1}$, which itself induces $\sigma_{p+1}\colon K(\partial_p)\rightarrow(N\Delta E_vG)_{p+1}$ for the same reasons as in \ref{prop homology EG}, plus the fact that $(d_0^h)^{p+2}s_{p+1}^h=(d_0^h)^{p+1}$. The lemma follows, since $\sigma_{p+1}$ is a splitting of the map $(NEG)_{q+1}\rightarrow K(\partial_q)$ induced by $\partial_{q+1}$.
\end{proof}

Hence $H_p(\Delta I_vG)$ is isomorphic to $H_{p-1}(\Delta \Omega_v G)$. Moreover, $H_p(\Delta C_vH_0^vG)=H_p^hH_0^vG$, so we obtain a long exact sequence

\begin{equation}\label{long exact sequence for q=1}
	\begin{tikzpicture}[baseline=(G)]
		\node(A) at (0,0){$\cdots$};
		\node(B) at (2,0){$H_{p-1}(\Delta\Omega_v G)$};
		\node(C) at (4.4,0){$H_p(\Delta G)$};
		\node(D) at (6.5,0){$H_p^hH_0^vG$};
		\node(E) at (8.9,0){$H_{p-2}(\Delta\Omega_v G)$};
		\node(F) at (10.8,0){$\cdots$};

		\node(G) at (0,-0.1){};

		\draw[\map] (A)--(B);
		\draw[\map] (B)--(C);
		\draw[\map] (C)--(D);
		\draw[\map] (D)--(E);
		\draw[\map] (E)--(F);
	\end{tikzpicture}
\end{equation}
We have $H_0^v(\Omega_v^qG)=H_q^v(G)$ by \ref{iso omega}. Applying \ref{long exact sequence for q=1} to $\Omega_v^qG$ instead of $G$, we obtain for every $q\geqslant0$ the long exact sequence
\[
	\begin{tikzpicture}
		\node(A) at (0,0){$\cdots$};
		\node(B) at (2.15,0){$H_{p-1}(\Delta\Omega_v^{q+1} G)$};
		\node(C) at (4.95,0){$H_p(\Delta\Omega_v^q G)$};
		\node(D) at (7.15,0){$H_p^hH_q^vG$};
		\node(E) at (9.75,0){$H_{p-2}(\Delta\Omega_v^{q+1} G)$};
		\node(F) at (11.8,0){$\cdots$};

		\draw[\map] (A)--(B);
		\draw[\map] (B)--(C);
		\draw[\map] (C)--(D);
		\draw[\map] (D)--(E);
		\draw[\map] (E)--(F);
	\end{tikzpicture}
\]
Let us define two bigraded objects $E$ and $D$ in $\C$ by setting:
\[
	E_{p,q}\coloneq
	\begin{cases}
		H_p^hH_q^vG & \text{if }p\geqslant0\text{ and }q\geqslant0 \\
		0           & \text{else}
	\end{cases}
\]
and
\[
	D_{p,q}\coloneq
	\begin{cases}
		H_p(\Delta\Omega_v^q G) & \text{if }p\geqslant0\text{ and }q\geqslant0 \\
		0                       & \text{if }p<0                                \\
		D_{p-1,q+1}             & \text{if }q<0
	\end{cases}
\]
We obtain an exact couple $\E$ of degree $2$:
\[
	\begin{tikzpicture}
		\node(A) at (0,0){$D$};
		\node(B) at (3,0){$D$};
		\node(C) at (1.5,-2){$E$};

		\draw[\map] (A)--(B) node[above,midway] {$f$};
		\draw[\map] (B)--(C) node[below right,midway] {$g$};
		\draw[\map] (C)--(A) node[below left,midway] {$h$};
	\end{tikzpicture}
\]

Here we set $f_{p,q}$ to be the identity map for $q<0$. To ensure that $\E$ is a normal exact couple, we need to recall some facts about abelian objects.

\begin{definition} \cite{Huq}\cite[Definition 1.5.3]{Borceux-Bourn}
	An object $X$ of $\C$ is \Def{abelian} if there exists a morphism $\mu\colon X\times X\rightarrow X$ such that the following commutes:
	\[
		\begin{tikzpicture}
			\node(A) at (0,0){$X$};
			\node(B) at (2.5,0){$X\times X$};
			\node(C) at (5,0){$X$};
			\node(D) at (2.5,-2){$X$};

			\draw[\map] (A)--(B) node[above,midway] {$(1,0)$};
			\draw[\map] (C)--(B) node[above,midway] {$(0,1)$};
			\draw[double,double distance=1mm] (A)--(D);
			\draw[\map] (B)--(D) node[right,midway] {$\mu$};
			\draw[double,double distance=1mm] (C)--(D);
		\end{tikzpicture}
	\]
\end{definition}

\begin{proposition}\label{prop sub quotient abelian}\cite[Theorem 4.2]{Gran:Central-Extensions}
	Any subobject of an abelian object is abelian. Any quotient object of an abelian object is abelian.\noproof
\end{proposition}

\begin{proposition}\label{prop sub abelian normal}\cite[Corollary 3.2.17]{Borceux-Bourn}
	Any subobject of an abelian object is normal.\noproof
\end{proposition}

\begin{theorem}\label{theorem abelian homology object}\cite[Theorem 5.5]{EveraertVdLinden}\cite[Proposition 3.1]{Bourn:Direct-Image}
	Let $G$ be a simplicial object in $\C$. For each $q\geqslant1$, the homology object $H_qG$ is an abelian object of $\C$.\noproof
\end{theorem}

Let us write $d\coloneq gh$. For each $(p,q)\neq(0,0)$, $E_{p,q}$ is an abelian object of $\C$ by Theorem \ref{theorem abelian homology object} and Proposition \ref{prop sub quotient abelian}, hence $d_{p,q}$ is normal. Moreover, $d_{0,0}$ has domain $E_{2,-1}=0$, and therefore is $0$.
Hence, $d$ is a normal map by \ref{prop sub abelian normal}, so we can derive $\E$ to get $\E'$. The same holds for  $\E'$ since $E'$ is a subquotient of $E$, and so on by induction, consequently $\E$ is a normal exact couple. Let $(E,d)$ be the spectral sequence that arises from $\E$, and $B^r$, $Z^r$ as in Proposition \ref{prop Br Zr}. Fix $p$, $q$ two non-negative integers, $n=p+q$, and write $H_n\coloneq D_{n,0}=H_{p+q}(\Delta G)$.	We have a filtration $(F_kH_n)_{0\leqslant k\leqslant n}$ of $H_n$ defined by
\begin{align*}
	F_kH_n & \coloneq I\Big(f_{n,0}f_{n-1,1}\cdots f_{k+1,n-k-1}\colon D_{k,n-k}\rightarrow H_n\Big) \\
	       & =I\Big((S^{k-n}f)\cdots(S^{-1}f)\Big)_{k,n-k}
\end{align*}
On one hand, for every $r>n-k+2$, the map $f_{k-2+r,n-k+2-r}=(S^{2-r}f)_{k,n-k}$ is the identity of $H_n$ because $n-k+2-r<0$. In particular, for every $r>q+2$, $(S^{2-r}f)_{p,q}$ is the identity of $H_n$, so Lemma \ref{lemma Br} implies that
\[
	B^\infty_{p,q}=B^{q+2}_{p,q}=gK(\varphi)
\]
where $\varphi=(S^{2-(q+2)}f)\cdots(S^{-2}f)(S^{-1}f)_{p,q}=f_{n,0}f_{n-1,1}\cdots f_{p+1,q-1}$. We have a commutative diagram
\begin{equation}\label{diagram with cokernel}
	\begin{tikzpicture}[baseline=(I)]
		\node(A) at (0,0){$K(\varphi f_{p,q})$};
		\node(B) at (3,0){$D_{p-1,q+1}$};
		\node(C) at (6,0){$F_{p-1}H_n$};
		\node(D) at (0,-2){$K(\varphi)$};
		\node(E) at (3,-2){$D_{p,q}$};
		\node(F) at (6,-2){$F_pH_n$};

		\node(G) at (0.5,-0.5){$\pullback$};

		\node(I) at (-1,-1.2){};

		\draw[\normalmono] (A)--(B);
		\draw[\normalepi] (B)--(C);
		\draw[\map] (A)--(D);
		\draw[\map] (B)--(E) node[left,midway] {$f_{p,q}$};
		\draw[\mono] (C)--(F);
		\draw[\normalmono] (D)--(E);
		\draw[\normalepi] (E)--(F);
	\end{tikzpicture}
\end{equation}

\begin{proposition}
	Taking the cokernels of the vertical maps in \eqref{diagram with cokernel} yields a short exact sequence
	\[
		\begin{tikzpicture}
			\node(A) at (0,0){$B^\infty_{p,q}$};
			\node(B) at (2.4,0){$I(g)_{p,q}$};
			\node(C) at (6,0){$C(F_{p-1}H_n\rightarrow F_pH_n)$};

			\draw[\normalmono] (A)--(B);
			\draw[\normalepi] (B)--(C);
		\end{tikzpicture}
	\]
\end{proposition}
\begin{proof}
	Write $\mu$ for the left-hand vertical map and $\nu$ for the right-hand vertical map in \eqref{diagram with cokernel}. Since $C(f_{p,q})=I(g)_{p,q}$, $C(\nu)$ is the pushout of $I(g)_{p,q}$ and $F_p{H_n}$ over $D_{p,q}$ by Lemma \ref{lemma pushout recognition}. Therefore we have a normal epimorphism $I(g)_{p,q}\norepi C(\nu)$ as a pushout of the normal epimorphism $D_{p,q}\norepi F_pH_n$. On the other side, $\mu$ being a pullback of the normal map $f_{p,q}$ gives a pullback square of normal monomorphisms
	\[
		\begin{tikzpicture}
			\node(A) at (0,0){$I(\mu)$};
			\node(B) at (3,0){$I(f_{p,q})$};
			\node(C) at (0,-2){$K(\varphi)$};
			\node(D) at (3,-2){$D_{p,q}$};

			\node(E) at (0.5,-0.5){$\pullback$};

			\draw[\normalmono] (A)--(B);
			\draw[\normalmono] (A)--(C);
			\draw[\normalmono] (B)--(D);
			\draw[\normalmono] (C)--(D);
		\end{tikzpicture}
	\]
	By Lemma \ref{lemma diexact}, taking the cokernels of the vertical maps gives a normal monomorphism $C(\mu)\normono I(g)_{p,q}$, which is the image of K($\varphi f_{p,q}$) in $I(g)_{p,q}$, hence $C(\mu)=gK(\varphi)=B^\infty_{p,q}$.
\end{proof}

On the other hand, the map $f_{p+1-r,q+r-2}=(S^{r-3}f)_{p-2,q+1}\colon D_{p-r,q+r-1}\rightarrow D_{p+1-r,q+r-2}$ is zero for every $r>p$, so we have
\begin{align*}
	S^2Th^{-1}I\Big(f(Sf)\cdots(S^{r-3}f)\Big)_{p,q} & =h^{-1}I\Big(f(Sf)\cdots(S^{r-3}f)\Big)_{p-2,q+1} \\
	                                                 & =h^{-1}(0)_{p-2,q+1}                              \\
	                                                 & =K(h)_{p-2,q+1}                                   \\
	                                                 & =S^2TK(h)_{p,q}                                   \\
	                                                 & =I(g)_{p,q}
\end{align*}
which by Lemma \ref{lemma Zr} gives
\[
	Z^\infty_{p,q}=Z^r_{p,q}=I(g)_{p,q}
\]
Thus we have $C(F_{p-1}H_n\rightarrow F_pH_n)\cong E^\infty_{p,q}$, and we proved:
\begin{theorem}\label{main theorem}
	If $G$ is a double simplicial object in a semi-abelian category, then there are spectral sequences
	\begin{align*}
		E^2_{p,q} & = H_p^h H_q^v G \Rightarrow H_{p+q}(\Delta G) \notag           \\
		E^2_{p,q} & = H_p^v H_q^h G \Rightarrow H_{p+q}(\Delta G) \tag*{$\square$}
	\end{align*}
\end{theorem}

\section{Appendix: Some Useful Lemmas}

In this section, we state some classical lemmas that are needed all along the text. These are straightforward to prove, and satisfied in any pointed category with kernels and cokernels, except for the last two.

\begin{lemma}\label{lemma kernel of pullback}
	For a commutative diagram
	\[
		\begin{tikzpicture}
			\node(A) at (0,0){$\bullet$};
			\node(B) at (2,0){$\bullet$};
			\node(C) at (0,-2){$\bullet$};
			\node(D) at (2,-2){$\bullet$};
			\node(E) at (4,0){$\bullet$};
			\node(F) at (4,-2){$\bullet$};
			\draw[\normalmono] (A)--(B) node[above,midway]{$\mathrm{ker}(f)$};
			\draw[\normalmono] (C)--(D) node[below,midway]{$\mathrm{ker}(g)$};
			\draw[\map] (A)--(C) node[left,midway]{$\gamma$};
			\draw[\map] (B)--(D) node[left,midway]{$u$};
			\draw[\map] (B)--(E) node[above,midway]{$f$};
			\draw[\map] (D)--(F) node[below,midway]{$g$};
			\draw[\map] (E)--(F) node[left,midway]{$v$};
		\end{tikzpicture}
	\]
	if the right-hand square is a pullback, then $\gamma$ is an isomorphism.\noproof
\end{lemma}

\begin{lemma}\label{lemma cokernel of pushout}
	For a commutative diagram
	\[
		\begin{tikzpicture}
			\node(A) at (0,0){$\bullet$};
			\node(B) at (2,0){$\bullet$};
			\node(C) at (0,-2){$\bullet$};
			\node(D) at (2,-2){$\bullet$};
			\node(E) at (4,0){$\bullet$};
			\node(F) at (4,-2){$\bullet$};
			\draw[\normalepi] (B)--(E) node[above,midway]{$\mathrm{coker}(f)$};
			\draw[\normalepi] (D)--(F) node[below,midway]{$\mathrm{coker}(g)$};
			\draw[\map] (A)--(C) node[left,midway]{$u$};
			\draw[\map] (B)--(D) node[left,midway]{$v$};
			\draw[\map] (A)--(B) node[above,midway]{$f$};
			\draw[\map] (C)--(D) node[below,midway]{$g$};
			\draw[\map] (E)--(F) node[left,midway]{$\gamma$};
		\end{tikzpicture}
	\]
	if the left-hand square is a pushout, then $\gamma$ is an isomorphism.\noproof
\end{lemma}

\begin{lemma}\label{lemma pullback recognition}
	Let $q\colon Y\rightarrow Z$, $q'\colon Y'\rightarrow Z'$ be two morphisms, $k=\ker(q)\colon X\normono Y$ and $k'\colon X'\rightarrowtail Y'$ a monomorphism such that $q'k'=0$, and suppose we have the following commutative diagram
	\[
		\begin{tikzpicture}
			\node(A) at (0,0){$X$};
			\node(B) at (2,0){$Y$};
			\node(C) at (0,-2){$X'$};
			\node(D) at (2,-2){$Y'$};
			\node(E) at (4,0){$Z$};
			\node(F) at (4,-2){$Z'$};
			\draw[\normalmono] (A)--(B) node[above,midway]{$k$};
			\draw[\mono] (C)--(D) node[below,midway]{$k'$};
			\draw[\map] (A)--(C) node[left,midway]{$u$};
			\draw[\map] (B)--(D) node[left,midway]{$v$};
			\draw[\map] (B)--(E) node[above,midway]{$q$};
			\draw[\map] (D)--(F) node[below,midway]{$q'$};
			\draw[\map] (E)--(F) node[left,midway]{$w$};
		\end{tikzpicture}
	\]
	If $w$ is a monomorphism, then the left-hand square is a pullback.\noproof
\end{lemma}

\begin{lemma}\label{lemma pushout recognition}
	Let $k\colon X\rightarrow Y$, $k'\colon X'\rightarrow Y'$ be two morphisms, $q'=\coker(k')\colon Y'\norepi Z'$ and $q\colon Y\twoheadrightarrow Z$ an epimorphism such that $qk=0$, and suppose we have the following commutative diagram
	\[
		\begin{tikzpicture}
			\node(A) at (0,0){$X$};
			\node(B) at (2,0){$Y$};
			\node(C) at (0,-2){$X'$};
			\node(D) at (2,-2){$Y'$};
			\node(E) at (4,0){$Z$};
			\node(F) at (4,-2){$Z'$};
			\draw[\map] (A)--(B) node[above,midway]{$k$};
			\draw[\map] (C)--(D) node[below,midway]{$k'$};
			\draw[\map] (A)--(C) node[left,midway]{$u$};
			\draw[\map] (B)--(D) node[left,midway]{$v$};
			\draw[\epi] (B)--(E) node[above,midway]{$q$};
			\draw[\normalepi] (D)--(F) node[below,midway]{$q'$};
			\draw[\map] (E)--(F) node[left,midway]{$w$};
		\end{tikzpicture}
	\]
	If $u$ is an epimorphism, then the right-hand square is a pushout.\noproof
\end{lemma}

\begin{lemma}\label{lemma kernel of composite}
	If $f$ and $g$ are composable morphisms, then the kernel of $gf$ is the pullback of the kernel of $g$ along $f$, \ie we have the formula $K(gf)=f^{-1}K(g)$. Dually, the cokernel of $gf$ is the pushout of the cokernel of $f$ along $g$.\noproof
\end{lemma}

\begin{lemma}\label{lemma pb with normal epis}
	Consider a commutative square
	\[
		\begin{tikzpicture}
			\node(A) at (0,0){$A$};
			\node(B) at (2,0){$B$};
			\node(C) at (0,-2){$C$};
			\node(D) at (2,-2){$D$};

			\draw[\map] (A)--(B) node[above,midway] {$f$};
			\draw[\map] (C)--(D) node[below,midway] {$g$};
			\draw[\map] (A)--(C);
			\draw[\map] (B)--(D);
		\end{tikzpicture}
	\]
	If it is a pullback and $f$ and $g$ are normal epimorphisms, then it is also a pushout. Dually, if it is a pushout and $f$ and $g$ are normal monomorphisms, then it is also a pullback.\noproof
\end{lemma}

\begin{lemma}\label{lemma pushout with epi}
	Consider a commutative diagram in which the outer rectangle is a pushout
	\[
		\begin{tikzpicture}
			\node(A) at (0,0){$\bullet$};
			\node(B) at (2,0){$\bullet$};
			\node(C) at (0,-2){$\bullet$};
			\node(D) at (2,-2){$\bullet$};
			\node(E) at (4,0){$\bullet$};
			\node(F) at (4,-2){$\bullet$};
			\draw[\map] (B)--(E) node[above,midway]{};
			\draw[\map] (D)--(F) node[below,midway]{};
			\draw[\map] (A)--(C) node[left,midway]{};
			\draw[\map] (B)--(D) node[left,midway]{};
			\draw[\map] (A)--(B) node[above,midway]{};
			\draw[\map] (C)--(D) node[below,midway]{$f$};
			\draw[\map] (E)--(F) node[left,midway]{};
		\end{tikzpicture}
	\]
	If $f$ is an epimorphism, then the right-hand square is a pushout.\noproof
\end{lemma}

The following lemma gives equivalent characterizations of what is called \textit{di-exactness} in a pointed category with kernels and cokernels. This notion is due to Van der Linden and Peschke in \cite{PVdL3}, and is true in every semi-abelian category. It is closely related to ``old-style'' axiom (SA*\,6) in \cite{Janelidze-Marki-Tholen}, which is well known---see, for instance, \cite{PVdL3} for an explicit proof---to be equivalent to Barr exactness in a homological category with finite colimits.

\begin{lemma} \label{lemma diexact}
	In a pointed category with kernels and cokernels, the following are equivalent:
	\begin{itemize}
		\item Let $e$ be a normal epimorphism, $m$ a normal monomorphism, and $f\coloneq em$. Then $f$ is a normal map.
		\item Suppose we have a commutative diagram with short exact rows
		      \[
			      \begin{tikzpicture}
				      \node(A) at (0,0){$\bullet$};
				      \node(B) at (2,0){$\bullet$};
				      \node(C) at (0,-2){$\bullet$};
				      \node(D) at (2,-2){$\bullet$};
				      \node(E) at (4,0){$\bullet$};
				      \node(F) at (4,-2){$\bullet$};
				      \draw[\normalmono] (A)--(B);
				      \draw[\normalmono] (C)--(D);
				      \draw[\map] (A)--(C) node[left,midway]{$u$};
				      \draw[\map] (B)--(D) node[left,midway]{$v$};
				      \draw[\normalepi] (B)--(E);
				      \draw[\normalepi] (D)--(F);
				      \draw[\map] (E)--(F) node[left,midway]{$w$};
			      \end{tikzpicture}
		      \]
		      If the left-hand square is a pullback of normal monomorphisms, then $w$ is a normal monomorphism. Dually, if the right-hand square is a pushout of normal epimorphisms, then $u$ is a normal epimorphism.
	\end{itemize}
	Both are satisfied in any semi-abelian category.\noproof
\end{lemma}

\section*{Acknowledgements}

I would like to express my sincere gratitude to my supervisor Tim Van der Linden for bringing me into this topic, and for his support and guidance throughout my research. I am also grateful to Rodrigue Haya Enriquez and David Forsman for the fruitful discussions about this work.

\providecommand{\noopsort}[1]{}
\providecommand{\bysame}{\leavevmode\hbox to3em{\hrulefill}\thinspace}
\providecommand{\MR}{\relax\ifhmode\unskip\space\fi MR }
\providecommand{\MRhref}[2]{%
	\href{http://www.ams.org/mathscinet-getitem?mr=#1}{#2}
}
\providecommand{\href}[2]{#2}

\end{document}